\documentclass[11pt]{amsart}
\usepackage{latexsym, amsmath, amssymb,amsthm,amsopn,amsfonts}
\usepackage{version}
\usepackage{epsfig,graphics,color,graphicx,graphpap}
\usepackage{amssymb}
\usepackage{todonotes}
\usepackage{listings}
\usepackage[normalem]{ulem}

\usepackage{subfigure}

\usepackage[citecolor=blue]{hyperref}
\usepackage[framemethod=tikz]{mdframed}

\numberwithin{equation}{section}
\newtheorem{theorem}{Theorem}[section]

\newtheorem{remark}{Remark}[section]

\newcommand{\LC}{\left(}
\newcommand{\RC}{\right)}

\newcommand{\LCB}{\left\{}
\newcommand{\RCB}{\right\}}

\newcommand{\R}{\mathbb R}

\definecolor{skyblue}{rgb}{0.85,0.85,1} 

\newcommand{\wt}{\widetilde}
\newcommand{\vE}{\mathbf{E}}
\newcommand{\vH}{\mathbf{H}}
\newcommand{\ve}{\mathbf{e}}
\newcommand{\vh}{\mathbf{h}}

\title[Near-nondiffracting accelerating beams]{Nonparaxial Near-Nondiffracting Accelerating Optical Beams}

\author[Lai]{Ru-Yu Lai}
\address{School of Mathematics, University of Minnesota, Minneapolis, MN 55455, USA}
\curraddr{}
\email{rylai@umn.edu }

\author[Zhou]{Ting Zhou}
\address{Department of Mathematics, Northeastern University, Boston, MA 02115, USA}
\curraddr{}
\email{t.zhou@neu.edu}
\thanks{  }

\date{\today}

\begin{document}

\begin{abstract}
    We show that new families of  accelerating and almost nondiffracting  beams (solutions) for Maxwell's equations can be constructed. These are complex geometrical optics (CGO) solutions to Maxwell's equations with nonlinear limiting Carleman weights. They have the form of wave packets that propagate along circular trajectories while almost preserving a transverse intensity profile. We also show similar waves constructed using the approach combining CGO solutions  and the Kelvin transform. 
\end{abstract}
\maketitle

\tableofcontents

\section{Introduction}
The research on nondiffracting accelerating optical beams has attracted considerable attention since the first study and observation of such beams reported in 2007 \cite{SC2, SC1}. These are optical wave packets that propagate along a curved trajectory while preserving its transverse amplitude structure. 
The first of such beams, known as the Airy beam, traces back to the context of quantum mechanics \cite{BB}, as a solution to the free-potential Schr\"odinger equation. In optics \cite{SC2, SC1}, the Airy beam is understood as a solution to the paraxial wave equation,
which is a good approximation of the propagation dynamics when the beam trajectory is limited to small (paraxial) angles $\sim10^\circ$. In this case, the Airy beam propagates along a parabolic trajectory while maintaining its intensity profile. 
The desirable feature of simultaneous shape-preserving and self-bending has invoked many intriguing applications including inducing curved plasma filaments \cite{PKMSC}, synthesizing versatile bullets of light \cite{CRCW}, carrying out autofocusing and supercontinuum experiments \cite{APM}, manipulating microparticles \cite{BMD} and so on.
However the spatial acceleration of a beam will make the bending angle continue increasing, eventually, the wave packet falls off into the non-paraxial regime and no longer nondiffracting.
In \cite{KBNS}, the authors found solutions to the Maxwell's equations in free space, that propagate along semicircular trajectories without losing the intensity of their main lobes after a large angle ($\sim 90^\circ$) bending. Roughly speaking, in the two dimensional transverse electric (TE) or transverse magnetic (TM) polarized cases, they examine the spherical harmonic expansion for the solution to the Helmholtz equation, splitting the integral for the Bessel function into two parts corresponding to both forward and backward propagations. The non-diffracting accelerating beam is the forward Bessel wave packets with apodization on the initial axis.  The three-dimensional accelerating beams in \cite{KBNS} are composed by a superposition of scalar solutions for the TE \slash TM polarization, multiplied by a plane wave in the direction perpendicular to the plane the accelerating trajectory lies in. 
Another approach to obtain 3D beams is to implement directly the splitting approach to the 3D spherical harmonic expansions in spherical coordinates (instead of cylindrical coordinates), as indicated in \cite{AMMK, AB}. 
This was generalized in \cite{AMMK} to find non-paraxial beams propagating along elliptic trajectories, that is, the magnitude of acceleration is no longer constant.

It is also explained in the previously mentioned physics literatures that the nondiffracting accelerating wave packets in free space, e.g., the Airy beam, is not contradicting the Ehrenfest's theorem, because the transverse intensity is not square integrable, hence the beam does not have a transverse center of mass. In experiments, see \cite{AB, SC2}, the localized beams with finite energy are induced by applying an exponential truncation (apodization). They still exhibit the key features over long distance propagation in spite of the fact that the center of gravity of these wave packets remains constant (an outcome of Ehrenfest's theorem) and diffraction eventually takes over. 

Non-diffracting accelerating beams are shown to exist in other type of media. In \cite{Schley}, these beams, as analytic solutions of Maxwell's equations with linear or nonlinear losses, propagate in absorbing media while maintaining their peak intensity. While the power such beams carry decays during propagation, the peak intensity and the structure of their main lobe region are maintained over large distances. 
Such loss-proof beams, when launched in vacuum or in lossless media, display exponential growth in peak intensity. This is achieved through the property of self-healing of non-diffracting beams, which allows energy transfer from the oscillating tail of the beam to the main lobe region. The self-healing properties, as a result of self transverse acceleration, is studied in \cite{BSDC}. In \cite{BNKS}, the idea is generalized to construct shape-preserving wave packets in curved space that propagate along non-geodesic trajectories. 

In this paper, we show existence of other accelerating and near-nondiffracting solutions for non-paraxial equations. More precisely, we construct the Complex Geometrical Optics (CGO) solutions with nonlinear Limiting Carleman Weights (LCW).

In the first approach, we consider the propagation in heterogeneous media. 
Let $\Omega$ be a bounded domain in $\mathbb{R}^3$ with smooth boundary.
We consider the time dependent Maxwell's equations
\[
        \nabla_x \times  \mathbf{E}(x,t)+\mu\displaystyle\frac{\partial\mathbf{H}(x,t)}{\partial t}=0, \qquad
        \nabla_x \times  \mathbf{H}(x,t) -\varepsilon\displaystyle\frac{\partial\mathbf{E}(x,t)}{\partial t}=\sigma\mathbf{E}(x,t),
\]
where $\varepsilon-\varepsilon_0, \mu-\mu_0, \sigma\in C_0^{2}(\Omega)$ for some positive constants $\mu_0$ and $\varepsilon_0$, and the corresponding time-harmonic Maxwell's equations of 
$$ 
    \vE(x)=\vE(x,t)e^{i\omega t},\qquad \vH(x)=\vH(x,t)e^{i\omega t},
$$
given by
\begin{equation} \label{eqn:max_th}
\nabla\times\vE-i\omega\mu\vH=0,\qquad 
\nabla\times\vH+i\omega\gamma\vE=0,
\end{equation}
where $\gamma=\varepsilon+i\sigma/\omega$ and $\omega>0$ is the angular frequency. 

The beam we construct is based on the so-called complex geometrical optics (CGO) solutions to Maxwell's equations \eqref{eqn:max_th}. These solutions were widely used in solving inverse problems arising from imaging modalities using electricity, electromagnetic waves and so on.  
Taking Electric Impedance Tomography (EIT) for example, the inverse problem is to reconstruct the conductivity function $\sigma(x)$ in a conductivity equation $\mbox{div}(\sigma\nabla u)=0$ in a medium $\Omega$ from the boundary Dirichlet-to-Neumann (voltage-to-current) map $u|_{\partial\Omega}\mapsto \nu\cdot\sigma\nabla u|_{\partial\Omega}$. This is also known as the Calder\'on problem and was generalized to an inverse problem for Maxwell's equations \eqref{eqn:max_th}, namely, to reconstruct the parameter set $(\mu,\varepsilon,\sigma)$ from boundary impedance map given by $\nu\times\vE|_{\partial\Omega}\mapsto\nu\times\vH|_{\partial\Omega}$. Here $\nu$ denotes the unit outer normal vector on the boundary $\partial\Omega$. These inverse problems were studied extensively (see \cite{U} for a detailed review) while through all the analysis, the CGO solutions introduced in \cite{SU1} plays a key role. An example of such solutions to the conductivity equation $\mbox{div}(\sigma\nabla u)=0$ is given by 
\begin{equation}\label{eqn:CGO_cond}
u=\sigma^{-1/2}e^{\zeta\cdot x}\left(1+\psi(x)\right),
\end{equation}
where $\zeta\in\mathbb C^n$ ($n$ is the spatial dimension) satisfying $\zeta\cdot\zeta=0$ and $\psi\in L^2$ satisfies a decaying property with respect to $|\zeta|\gg1$. More precisely, 
\begin{equation}\label{eqn:CGO_small}\|\psi\|_{L^2}\leq \frac{C}{|\zeta|}.\end{equation}
Let $\tau>0$ and $\alpha>0$. In $\mathbb R^3$, set $\zeta=(\tau,i\sqrt{\tau^2-\alpha^2}, i\alpha)^t$  (note that $\zeta\cdot\zeta=0$ and $|\zeta|=\sqrt 2\tau$).  
Assume that $\sigma$ is independent of $x_3$. The above decaying property implies that $e^{-i\omega t}u$, with $u$ given by \eqref{eqn:CGO_cond}, is roughly a plane wave packet propagating nondiffractingly along $\hat e_3$ direction for $\tau$ large. The transverse profile of the beam is oscillatory in $\hat e_2$ and exponential in $\hat e_1$. (Compared to the Airy beam, the tail of the CGO plane beam does not decay with respect to $x_2$). 
Note that the CGO solution is a high energy near-nondiffracting wave packet for $\tau$ is large. 

Using the Liouville transform $u=\sigma^{-1/2}v$, the conductivity equation $\mbox{div}(\sigma\nabla u)=0$ is reduced to the Schr\"odinger equation $\Delta v+q(x)v=0$ with potential $q=- \Delta\sigma^{1/2}/ \sigma^{1/2}\in L^\infty(\Omega)$. The exsitence of a global CGO solution $u$ of the form \eqref{eqn:CGO_cond} is equivalent to solving an equation of $\psi$ in $\R^3$ given by
\begin{equation}\label{eqn:CGO_psi}
\Delta\psi+2\zeta\cdot\nabla\psi+q(x)\psi=-q(x).
\end{equation} 
The Faddeev's kernel defines an inverse of $(\Delta+2\zeta\cdot\nabla)$ by
\[G_\zeta(f)=\mathcal F^{-1}\left(\frac{\mathcal F f}{-|\xi|^2+2i\zeta\cdot\xi}\right),\]
where $\mathcal F$ and $\mathcal F^{-1}$ denote the Fourier transform and its inverse, respectively. Furthermore, it is shown in \cite{SU1} that for $\tau$ large enough,
\begin{equation}\label{eqn:Fadd_est}
\|G_\zeta\|_{L^2_{\delta}\rightarrow L^2_{-\delta}}\lesssim \tau^{-1}
\end{equation}
for $\delta>-\frac{1}{2}$, where $L^2_\delta$ is the closure of $C_0^\infty(\R^n)$ with respect to the weighted norm
$\|u\|_{L^2_\delta}=\|(1+|x|^2)^{\delta/2}u\|_{L^2}$. Then the existence of $\psi$ to \eqref{eqn:CGO_psi} and the estimate \eqref{eqn:CGO_small} are corollary of \eqref{eqn:Fadd_est} and the fact that $q(x)$ can be viewed as a compactly supported $L^\infty(\R^n)$ function. 

For Maxwell's equations, by introducing two auxiliary scalar fields $\Phi$ and $\Psi$ and a Liouville type rescaling, the first order system can be reduced to a Dirac system $(P-k+W)Y=0$. A solution $Y$ of the Dirac system gives a solution to the original Maxwell system iff $\Phi=\Psi=0$. The reduction to the Schr\"odinger equation $(-\Delta-k^2+Q)Z=0$ is then due to $(P-k+W)(P+k-W^t)=-\Delta-k^2+Q$. (Detail of the reduction is outlined in Section \ref{section3D}.) This reduction was first introduced in \cite{OS} and became a standard step for construction of CGO for Maxwell's equations ever since. Once the CGO solutions are constructed for the Schr\"odinger equation, a uniqueness result is required to show that the scalar potentials $\Phi$ and $\Psi$ are vanishing. As a result, one obtains a CGO solution to Maxwell's equations of the form
\[ \vE=\gamma^{-1/2}e^{\zeta\cdot x}\tau\left(\eta+O_{\tau}(\tau^{-1})\right),\quad \vH=\mu^{-1/2}e^{\zeta\cdot x}\tau\left(\theta+O_\tau(\tau^{-1})\right)
\]
for $\tau$ sufficiently large, where $\eta$ and $\theta$ are two constant vectors in $x$ and bounded in $\tau$. Similar to the scalar case, under the assumption that the parameters are functions only depending on transverse variables, these solutions are high energy near-nondiffracting beams.

In above construction, replacing the linear phase $x\cdot\zeta$ by a nonlinear one, denoted by $\varphi(x)$, it is shown in \cite{KSjU} that solutions of CGO type can be constructed to the Schr\"odinger equations using Carleman estimates, and also to the magnetic Schr\"odinger equations as seen in \cite{FKSjU}. These admissible phases are called Limiting Carleman Weights (LCW). There are only handful LCW in three dimensions while all analytical functions in $z=x_1+ix_2$ can be used as an LCW in two dimensions. With nonlinear phases, it allows to construct solutions propagating along a curved surface (accelerate) while almost preserving its intensity profile (near-nondiffracting), usually due to a smallness estimate similar to \eqref{eqn:CGO_small}. In \cite{FKSjU, KSjU}, the CGO with the LCW $\varphi(x)=-\log|x|$ was used to solve the Calder\'on problem with partial measurements. In \cite{KSU}, the LCW $\varphi(x)=x_1$ was used to construct solutions on Riemannian manifolds with a family of admissible metrics, modeling the anisotropic materials.

In order to generalize this to construct CGO solutions to Maxwell's equations with nonlinear LCW, it requires uniqueness of the solution to the reduced Schr\"ondinger equation. In \cite{NS}, the uniqueness is obtained on a bounded domain by projecting to a function space with fixed boundary conditions. In \cite{KSU}, different approach is adopted to obtain the uniqueness, by applying the original Fourier analysis in \cite{SU1} to $x_1$ direction. The CGO solutions with either $\varphi(x)=-\log|x|$ or $\varphi(x)=x_1$ can be put into a unified framework in cylindrical coordinate $(x_1, r,\theta)$. In particular, we show that the dependence of the CGO solutions on $\theta$ can be mainly $e^{i\rho\theta}$, so that we obtain the bending electromagnetic beams.


In the second framework, we adopt a very different approach that is based on a special transformation, known as the Kelvin transform. In $\R^3$, the three-dimensional Kevin transform $K$ is a reflection with respect to a sphere, which maps hyperplanes to spheres that pass the origin. We exploit the transformation law by $K$. Moreover, we show that a non-diffracting beam that propagates along straight lines, such as a CGO solution with a linear phase, can be ``pushed forward" to generate a beam that accelerates along the circular trajectory. The acceleration is strong enough to shift the energy from the tail to the main lobe, over-compensating the intensity of the first lobe while propagating. This effect is independent of the background medium, homogeneous or heterogeneous, lossless or lossy.


The paper is organized as follows. In Section 2, we show the main steps to construct the CGO solutions to Maxwell's equations based on Carleman estimates. Section 3 is devoted to the construction based on Kelvin transform. Both sections are complemented with the demonstration of corresponding solutions.

\textbf{Acknowledgements.} Both authors thank professor Gunther Uhlmann for suggesting this problem and for useful discussions.
The research of the first author is partly supported by the AMS-Simons Travel Grant. The research of the second author is supported by the NSF grant DMS-1501049.

\section{CGO accelerating beams for Maxwell's equations}\label{section3D}
In this section, we aim to construct the almost diffraction-free beams for the inhomogeneous Maxwell's equations in dimension three. We first discuss the construction of such beams with LCW $\varphi(x) = -x_1$. The second construction is to apply another nonlinear LCW $\varphi(x) = -\log|x|$.
Numerical demonstrations of these accelerating beams are also presented.
 

For completeness, we first include a reduction, introduced in \cite{OS}, that transforms the Maxwell's system to the vectorial Schr\"odinger equation. 
  
\subsection{Reduction to the Schr\"odinger equation}
Let $\Omega$ be a bounded domain in $\mathbb{R}^3$ with smooth boundary. We consider the time-harmonic Maxwell's equations where $\varepsilon-\varepsilon_0, \mu-\mu_0, \sigma\in C_0^{2	}(\overline\Omega)$ for some positive constant $\mu_0$ and $\varepsilon_0$,
\begin{equation}\label{r:max}
        \nabla \times  \mathbf{E} -i\omega \mu  \mathbf{H}=0,\quad
        \nabla \times  \mathbf{H}+i\omega\gamma\mathbf{E}=0, 
\end{equation}
where $\gamma=\varepsilon+i\sigma/\omega$ and $\omega>0$ in $\overline\Omega$.
We remark here that the asymptotic-to-constant assumption $\mu-\mu_0, \gamma-\varepsilon_0\in C_0^2(\overline\Omega)$ is without loss of generality since any smooth parameters $\mu,\gamma$ in a bounded domain can be extended to satisfy it in a larger domain that contains the propagating region.

From (\ref{r:max}), one has the following compatibility conditions for $\mathbf{E}$ and $\mathbf{H}$
\begin{equation}\label{2:max}
       \nabla \cdot(\gamma \mathbf{E})=0, \quad
        \nabla\cdot (\mu \mathbf{H})=0.
      \end{equation}
There are eight equations in \eqref{r:max} and \eqref{2:max} for six unknowns, components of $\vE$ and $\vH$.  It allows us to augment two scalar potentials $\Phi$ and $\Psi$ to obtain
\begin{equation}\label{eqn:Max_aug}
\left\{
      \begin{split}
       D\cdot \mathbf{E}+D\alpha\cdot \mathbf{E}-\omega\mu\Phi=0, \\
      D \times  \mathbf{E} -\omega \mu \mathbf{H}+\gamma^{-1}D(\gamma\Psi)=0, \\
       D\cdot  \mathbf{H} +D\beta\cdot \mathbf{H}-\omega\gamma\Psi=0, \\
        -D \times  \mathbf{H}-\omega\gamma \mathbf{E}+\mu^{-1}D(\mu\Phi)=0,\\
      \end{split}
    \right.
\end{equation}
where $D=-i\nabla$, $\alpha=\log\gamma$ (the principal branch) and $\beta=\log\mu$. Set the unknown to be the eight vector $X=(\Phi,\vH^t,\Psi,\vE^t)^t$. We can write \eqref{eqn:Max_aug} as
\begin{equation}\label{eqn:dirac}
(P+V)X=0,
\end{equation}
where $P$ is a first order Dirac operator given by
\begin{equation}\label{eqn:PofD}
P=\left(\begin{array}{c|c} & \begin{array}{cc} & D\cdot \\D & D\times\end{array} \\\hline \begin{array}{cc} & D\cdot \\D & -D\times\end{array} & \end{array}\right)
\end{equation} 
and 
\[V=\left(\begin{array}{c|c}\begin{array}{cc}-\omega\mu &  \\ & -\omega\mu I_3\end{array} & \begin{array}{cc} & D\alpha\cdot \\D\alpha & \end{array} \\\hline \begin{array}{cc} & D\beta\cdot \\D\beta & \end{array} & \begin{array}{cc}-\omega\gamma &  \\ & -\omega\gamma I_3\end{array}\end{array}\right).
\]
Here $I_j$ is the $(j\times j)$-identity matrix.
Note that $P^2=-\Delta I_8$. Moreover, we mention a fact, later becoming very important, that Maxwell's equations \eqref{r:max} is equivalent to the Dirac system \eqref{eqn:dirac} if and only if $\Phi=\Psi=0$. 

Throughout the paper we also use the notation 
$$
X:=\left(x^{(1)}, {X^{(1)}}^t, x^{(2)}, {X^{(2)}}^t\right)^t
$$ 
for all eight-vectors $X\in\mathbb C^8$ where lower cases $x^{(j)}$ are scalar and upper cases $X^{(j)}$ are three-vectors.   Now we apply a Liouville type of rescaling by letting 
$$
     Y
     =\left(         \begin{array}{c|c}
           \mu^{1/2}I_4 &  \\ \hline
            & \gamma^{1/2} I_4\\
         \end{array}
       \right)X.
$$
Then we have
\begin{equation}\label{eqn:dirac_rescale}
(P-k+W)Y=0,
\end{equation}
where $k=\omega\mu_0^{1/2}\varepsilon_0^{1/2}$ and 
\[
W=-(\kappa-k) I_8+\frac{1}{2}
\left(
\begin{array}{cc|cc}
   &  &  & D\alpha\cdot \\
  & &   D\alpha & -D\alpha\times\\
  \hline
   & D\beta\cdot & & \\
    D\beta &  D\beta\times &  & \\
\end{array}
       \right)
\]
with $\kappa=\omega\mu^{1/2}\gamma^{1/2}$.

One can easily verify that with the rescaling, the first order terms in $(P-k+W)(P+k-W^t)$ cancel each other and we obtain 
\begin{equation}\label{P1}
   (P-k+W)(P+k-W^t)=-\Delta-k^2+Q,
\end{equation}
where the matrix potential $Q$ is an $(L^\infty(\R^3))^{8\times 8}$-valued potential with compact support in $\Omega$, whose entries involve the parameters $(\mu,\gamma)$ and their first and second derivatives. The form of $Q$ is not crucial in this construction so we omit writing the explicit formula. For readers who are interesting in the expression of $Q$, it can be found, for example, in \cite{KSU}. 
We will also need the following relations
\[
   (P+k-W^t)(P-k+W)=-\Delta -k^2+\wt Q,
\]
where $\wt Q$ shares the same property as $Q$. An extra fact of $\wt Q$ we will take advantage of is that the first and the fifth rows are diagonal. Here $W^t$ denotes the transpose of $W$.

\subsection{CGO solutions for the Schr\"odinger equation}

In this part, we outline the construction steps of CGO solutions to the Schr\"odinger equations based on the Carleman estimate as in \cite{FKSjU, KSjU}. However, the general construction does not provide uniqueness that is necessary for translating to solutions to Maxwell's equations. Therefore, our argument will bifurcate slightly for two different choices of limiting Carleman weights in order to obtain uniqueness respectively. 

A real smooth function $\varphi$ on an open set $\wt\Omega$ is a LCW if it has a non-vanishing gradient on $\wt\Omega$ and if it satisfies the condition
\[\langle\varphi''\nabla\varphi,\nabla\varphi\rangle+\langle\varphi''\xi,\xi\rangle=0\quad\mbox{ when }\quad |\xi|^2=|\nabla\varphi|^2\;\mbox{ and }\;\xi\cdot\nabla\varphi=0.\]
Roughly speaking, this is the H\"ormander condition that guarantees the solvability of the semi-classical conjugate operator $e^{\tau\varphi}(-\tau^{-2}\Delta)e^{-\tau\varphi}$ for both $\varphi$ and $-\varphi$. The tool used to obtain it is the Carleman estimate, see \cite{FKSjU, KSjU}. 

We are looking for the complex geometrical optics solutions of the form
\begin{equation}\label{eqn:CSW_Z_Sch}
     Z(x)=e^{\tau(\varphi+i\psi)}\left(A(x)+R(x)\right),
\end{equation}
where $R$ is neglectable compared to $A$ in the semi-classical sense. Here $\psi$ satisfies the eikonal equation, read as
\begin{equation}\label{eqn:CSW_eikonal}|\nabla\psi|^2=|\nabla\varphi|^2,\quad \nabla\psi\cdot\nabla\varphi=0.\end{equation}
With such $\varphi$ and $\psi$, we denote 
\[\mathcal L_{\varphi+i\psi}:=e^{-\tau(\varphi+i\psi)}(-\Delta-k^2+Q)e^{\tau(\varphi+i\psi)}.\]
Then we have 
\begin{align*}
    \mathcal L_{\varphi+i\psi}R
    & = -\mathcal L_{\varphi+i\psi}A\\
    & = (\Delta+k^2-Q)A+\tau\left[2\nabla(\varphi+i\psi)\cdot\nabla+\Delta(\varphi+i\psi)\right]A.
\end{align*} 
For the remainder term $R$ to satisfy the decaying condition with respect to $\tau$, we ask $A$ to satisfy the transport equation 
\begin{equation}\label{eqn:CSW_A}
    \left[2\nabla(\varphi+i\psi)\cdot\nabla+\Delta(\varphi+i\psi)\right]A=0.\end{equation}
Suppose $A$ is a $C^2$ solution (they exist as shown in the cases below). Also, suppose $\Omega\subset\subset\wt\Omega$ and $Q\in L^\infty(\Omega)$. Then by the Proposition 2.4 in \cite{FKSjU}, for $\tau$ large enough, there exists an $R\in H^1(\Omega)$ satisfying 
\begin{equation}\label{eqn:CSW_R_smallness}
    \|R\|_{L^2(\Omega)}+\tau^{-1}\|\nabla R\|_{L^2(\Omega)}\leq C\tau^{-1}\|(\Delta+k^2-Q)A\|_{L^2(\Omega)}\end{equation}
for some constant $C$ independent of $\tau$. 

Now our attention switches to the phase $\varphi+i\psi$ and vector field $A$ in order to obtain near non-diffracting accelerating solutions in $\Omega$. We consider the phase in terms of $z=x_1+ir$ where $(x_1, r, \theta)$ denotes the cylindrical coordinate for $x=(x_1, x')$ with $x'=(x_2, x_3)=(r,\theta)$, polar coordinates in the $x'$ variable. Set $\varphi+i\psi=l(z)$ where $l$ is a $C^1$ function to be specified later. Then the transport equation \eqref{eqn:CSW_A} reads
\[l'(z)\left(2\frac\partial{\partial\overline z}-\frac 1{z-\overline z}\right)A=0.\]
We can then choose 
\begin{equation}\label{eqn:CSW_A_lambda}
    A(x)=(z-\overline z)^{-\frac12}g(\theta)=\frac {e^{i\lambda z}}{\sqrt{2ir}}g(\theta),
\end{equation}
where $\lambda\geq0$ is a constant and $g(\theta)$ is an arbitrary vector function of $\theta$. 
Then the dominant term of $Z$ as $\tau\gg1$ is given by
    \[e^{\tau(\varphi+i\psi)}A=\frac{e^{\tau l(x_1+ir)+i\lambda(x_1+ir)}}{\sqrt{2ir}}g(\theta)\]
which propagates non-diffractingly and along a circle if we choose $g(\theta)$ appropriately. 
Similarly, one can choose $\varphi+i\psi=l(\overline z)$. Then \eqref{eqn:CSW_A} reads 
$$l'(\overline z)\left(2\frac{\partial}{\partial z}+\frac1{z-\overline z}\right)A=0$$ 
and the solution is $ (2ir)^{-1/2} e^{i\lambda\overline z}g(\theta)$.\\

The choice of $l(\cdot)$ has to be such that $\varphi$ and $\psi$ satisfy respective conditions of LCW and \eqref{eqn:CSW_eikonal}. 
In three dimensional euclidean space, it is mentioned in \cite{FKSU, KS} that there are only six LCWs up to translation and scaling, among which $\varphi(x)=x_1$ and $\varphi(x)=-\log|x|$ only depend on $x_1$ and $r$, hence $z$. 

In \cite{KSU}, $\varphi=-x_1=\Re (-z)$ is used. It is not hard to verify that $\psi=-r$ satisfies the eikonal equation by writing the gradient as 
$$\nabla=\left(\partial_{x_1}, ~\cos\theta\partial_r-\frac{\sin\theta}r\partial_\theta, ~\sin\theta\partial_r+\frac{\cos\theta}r\partial_\theta\right)^t. $$
This is the case $l(z)=-z$. In order to show uniqueness, an approach different from Carleman estimate was taken in \cite{KSU} by taking advantage of that the phase is linear in $x_1$. Suppose $\tau^2+k^2 \notin \mbox{Spec}~(-\Delta_{x'})$. Using the Fourier decomposition of $-\Delta_{x'}$ with imposed zero Dirichlet condition on the transversal domain $\{x'~|~(x_1, x')\in \Omega\}$, this is reduced to solving $-\partial_{x_1}^2$ for complex geometrical optics eigen-modes. Since the phase is linear in $x_1$, this can be achieved by simulating the direct analysis for the linear phase case as in \cite{SU1}, which carries a uniqueness result globally in $x_1$. That is why the domain under consideration in this case is assumed to be cylindrical, for example, $T=\R\times B(0,R)$ where $B(0,R)$ denotes the disc in $\R^2$ centered at $0$ with radius $R$. Then the space in the norm estimate \eqref{eqn:CSW_R_smallness} is replaced by $L^2_\delta(T)$ or $H^s_\delta(T)$ defined using $x_1$-weighted norms
    \begin{equation}\label{eqn:weight_norm}
    \begin{split}
        \|u\|_{L^2_\delta (T)} &= \|\langle x_1 \rangle^\delta u\|_{L^2(T)},\\
         \|u\|_{H^s_\delta (T)} &= \|\langle x_1 \rangle^\delta u\|_{H^s(T)}.
         \end{split}
    \end{equation}
We also define the spaces
    \begin{align*}
        & H^1_{\delta, 0}(T) = \LCB u \in H^1_\delta(T): u|_{\R\times \partial B_R} = 0 \RCB,\\
        & H^1_{-\infty, 0}(T) = \bigcup_{\delta\in \R} H^1_{\delta, 0}(T).
    \end{align*}
Then we have the inverse of the conjugate operator $-\Delta_\varphi:=e^{-\tau\varphi}(-\Delta)e^{\tau\varphi}$
\[G_\tau: L^2_\delta(T)\rightarrow H^2_{-\delta}(T)\cap H^1_{-\delta,0}(T)\]
satisfies
\[\|G_\tau\|_{L^2_{\delta}(T)\rightarrow H^s_{-\delta}(T)}\leq C\tau^{s-1}\]
for $\delta>1/2$ and $0\leq s\leq 2$. However, if $\lambda>0$ in the choice of $A$ given by \eqref{eqn:CSW_A_lambda}, then $(\Delta+k^2-Q)A$, the right hand side of the equation for $R$, is not in $L^2_\delta(T)$ (not enough decay in $x_1$). In \cite{KSU}, it is shown in Proposition 5.1 that $G_\tau$ can be extended to include functions with special dependence in $x_1$ such as $e^{i\lambda x_1}$ on the right hand side, which takes care of this problem. Therefore, there exists a unique $R\in H^2_{-\delta}(T)\cap H^1_{-\delta,0}(T)$. Moreover, it satisfies
\begin{equation}\label{eqn:CSW_R_T_smallness}
    \|R\|_{H^s_{-\delta}(T)}\leq C\tau^{-1}\|(\Delta+k^2-Q)A\|_{L^2_{-\delta}(T)}.
    \end{equation}
In the case that $Q\in L^\infty(T)$ with compact support, we have $R\in H^1_{-\infty,0}(T)$. \\

On the other hand, in \cite{FKSjU, KSjU}, the authors used $\varphi(x)=-\log|x|$ for the case that $\Omega$ does not contain the origin. For our purpose, we consider $\Omega\subset\subset \R^3_+ \cap T$ where $\R^3_+$ denotes the upper half plane corresponding to $\theta\in (0,\pi)$. 
The corresponding $\psi(x)$ can be chosen as the angle formed by the vector $x$ and $x_1$ axis. In terms of $z$, we have
\[\varphi(x)=-\log|z|=-\Re (\log\overline z),\quad \psi(x)=\arctan\frac{\Im z}{\Re z}=-\Im(\log \overline z).\]
Therefore, $\varphi+i\psi=l(\overline z)=-\log \overline z$.

To address the uniqueness of the solution to the Schr\"odinger equation, one can adopt the orthogonal projection technique in \cite{NS} onto a subspace of $L^2(\Omega)$ with specified boundary condition adjusted to suit our case. Then we have $G_\tau=-\Delta_\varphi^{-1}: L^2(\Omega)\rightarrow H^2(\Omega)$ with $\|G_\tau\|_{L^2(\Omega)\rightarrow H^s(\Omega)}\leq C\tau^{s-1}$ where $-\Delta_{\varphi}:=e^{-\tau\varphi}(-\Delta)e^{\tau\varphi}$.

\begin{remark}\label{rmk:linear}
Here we would like to comment also on the case of linear phase $\varphi+i\psi=\zeta_0\cdot x$ in the Cartesian coordinate $(x_1, x_2, x_3)$ with $\zeta_0\in\mathbb C^3$, corresponding to another LCW, $\varphi(x)=\Re\zeta_0$. Then the eikonal equation gives $\Re\zeta_0\cdot \Im \zeta_0=0$ and $|\Re \zeta_0|=|\Im\zeta_0|$, or equivalently $\zeta_0\cdot\zeta_0=0$. The transport equation (\ref{eqn:CSW_A}) becomes $2\zeta_0\cdot \nabla A=0$. Hence $A$ is chosen to be a constant vector. The invertibility of the conjugate operator $-\Delta-2\tau\zeta_0\cdot\nabla$ relies on the direct Fourier analysis, as  in \cite {SU1},  in the whole $\R^3$ for uniqueness. The operator $G_\tau: L^2_\delta(\R^3)\rightarrow H^2_{-\delta}(\R^3)$ satisfies
    \[\|G_\tau\|_{L^2_{\delta}(\R^3)\rightarrow H^s_{-\delta}(\R^3)}\leq C\tau^{s-1}\]
for $\delta>1/2$ and $0\leq s\leq 2$, where $L^2_\delta(\R^3)$ and $H^s_\delta(\R^3)$ are defined by the norms \eqref{eqn:weight_norm} with $\langle x_1\rangle$ replaced by $\langle x\rangle$. 

Since our Schr\"odinger equations here are more of the Helmholtz type with wave number $k$. The phase $\tau\zeta_0$ above can be replaced by $\zeta\in\mathbb C^3$ with $\zeta\cdot\zeta=-k^2$ and $|\zeta|=\tau$. 
\end{remark}

\begin{remark}
The CGO solutions constructed above are the exact solutions of the Schr\"odinger equations in the inhomogeneous space. These solutions will be applied to construct the solutions to Maxwell's equations below and more importantly, they possess near-nondiffracting property as $|\zeta|$ is large. Besides, these solutions have implications for the study of other wave system in nature.  
\end{remark}

\subsection{CGO solutions for Maxwell's equations} 

In this part, we compute the CGO solutions to the original Maxwell's equations using the CGO solutions \eqref{eqn:CSW_Z_Sch} to the reduced Schr\"odinger equation and their uniqueness for the two LCW in cylindrical coordinate, in order to obtain the near non-diffracting accelerating beams.

Now the corresponding CGO solutions to the Dirac system $(P-k+W)Y=0$ are given by 
\begin{align*} 
    Y  &=(P+k-W^t)e^{\tau(\varphi+i\psi)}(A+R)=e^{\tau(\varphi+i\psi)}(B+S) \\
       &=: (y^{(1)}, Y^{(1)}, y^{(2)}, Y^{(2)} )^t,
\end{align*}
where 
\begin{equation}\label{eqn:CSW_B_S}\begin{split}
    B&=\tau P(D(\varphi+i\psi))A+(P+k)A\\
    &=: (b^{(1)}, B^{(1)}, b^{(2)}, B^{(2)} )^t ,\\
    S&=-W^tA+\tau P(D(\varphi+i\psi))R +(P+k-W^t)R\\
    &=: (s^{(1)}, S^{(1)}, s^{(2)}, S^{(2)} )^t .
\end{split}\end{equation}

We recall that for the corresponding 
    \[\vE=\gamma^{-1/2}Y^{(2)},\quad \vH=\mu^{-1/2}Y^{(1)}\]
to be the solution to Maxwell's equations, one must have $y^{(1)}=y^{(2)}=0$. In general, the strategy is to choose a proper 
$$
    A = (a^{(1)}, A^{(1)}, a^{(2)}, A^{(2)} )^t 
$$ 
such that the corresponding scalars of $B$, given by
\begin{equation}\label{eqn:CSW_b_12}
\begin{split}
    b^{(1)}&=\tau D(\varphi+i\psi)\cdot A^{(2)}+D\cdot A^{(2)}+k a^{(1)}\\
    b^{(2)}&=\tau D(\varphi+i\psi)\cdot A^{(1)}+D\cdot A^{(1)}+ka^{(2)}
\end{split}
\end{equation}
satisfy $b^{(1)}=b^{(2)}=0$. 
This is motivated by the fact that $Y$ satisfies the other Schr\"odinger equation 
    \[(P+k-W^t)(P-k+W)Y=(-\Delta -k^2+\wt Q)Y=0,\]
which implies
    \[(-\Delta-k^2+\wt q_j)y^{(j)}=0,\qquad j=1,2,\]
where $\wt q_1$ and $\wt q_2$ are diagonal components of the first and fifth row of $\wt Q$, given by
\[\begin{split}
          \wt q_1&=-(\kappa^2-k^2)-\frac{1}{2}\Delta\beta-\frac{1}{4}D\beta\cdot D\beta,\\
          \wt q_2&=-(\kappa^2-k^2)-\frac{1}{2}\Delta\alpha-\frac{1}{4}D\alpha\cdot D\alpha.
\end{split}\]  
They are compactly supported in $\Omega$.

To obtain the solution to Maxwell's equations, we start with the case on the cylinder $T$ such that $\Omega\subset\subset T$ with $\varphi+i\psi=-z$. Note that $W\in C^1_0(\Omega)$. 
Then we have

\begin{theorem}\label{cgomax}
     Let $k\geq 0$, $\lambda>0$ be constant and let $\delta > 1/2$. There exists $\tau_0\geq 1$ such that when
       $$
            |\tau|\geq \tau_0\ \hbox{and}\ \ \tau^2+k^2 \notin \mbox{Spec}~(-\Delta_{x'}) 
       $$
       we have that given smooth functions $\chi_1(\theta)$ and $\chi_2(\theta)$, there exists a solution to Maxwell's equations \eqref{r:max} in $T$ of the form
       \begin{equation}\label{cylinerMax}
       \begin{split}
     \vE(x_1,r,\theta) &= \tau\gamma^{-1/2} {{{e^{(-\tau+i\lambda)(x_1+ir)}}\over \sqrt{2ir}}}   \chi_1(\theta)
    \left(
    \begin{array}{c}
        -i\\
        \cos\theta\\
        \sin\theta\\
    \end{array} 
    \right) + O_{\tau}(1),\\
     \vH(x_1,r,\theta) &= \tau \mu^{-1/2} {{e^{(-\tau+i\lambda)(x_1+ir)}}\over \sqrt{2ir}} \chi_2(\theta)
        \left(
        \begin{array}{c}
            -i\\
            \cos\theta\\
            \sin\theta\\
        \end{array} 
        \right) + O_{\tau}(1),
\end{split}\end{equation}
where $O_\tau(1)$ stands for a function whose $L_{-\delta}^2(T)$, hence $L^2(\Omega)$ norm is of order $O(1)$ as $\tau\rightarrow\infty$. 
\end{theorem}

\begin{proof}     
Since $\varphi+i\psi=-z$, we have $D(\varphi+i\psi)=(i, -\cos\theta, -\sin\theta)^t$. Suppose $A$ is given by \eqref{eqn:CSW_A_lambda}.
Denote $g(\theta):=(g^{(1)}, {G^{(1)}}^t, g^{(2)}, {G^{(2)}}^t)^t$ the same convention of other eight-vectors. 
Then 
\[
\begin{split}
b^{(1)}=&\frac{e^{i\lambda z}}{\sqrt{2ir}}\left[\left(\begin{array}{c}i\tau+\lambda \\(-\tau+i\lambda+\frac i{2r})\cos\theta+\frac i r\sin\theta\partial_\theta \\(-\tau+i\lambda+\frac i{2r})\sin\theta-\frac i r\cos\theta\partial_\theta\end{array}\right)\cdot G^{(2)}(\theta)+kg^{(1)}(\theta)\right],\\
b^{(2)}=&\frac{e^{i\lambda z}}{\sqrt{2ir}}\left[\left(\begin{array}{c}i\tau+\lambda \\(-\tau+i\lambda+\frac i{2r})\cos\theta+\frac i r\sin\theta\partial_\theta \\(-\tau+i\lambda+\frac i{2r})\sin\theta-\frac i r\cos\theta\partial_\theta\end{array}\right)\cdot G^{(1)}(\theta)+kg^{(2)}(\theta)\right].
\end{split}
\]
Given $\chi_1(\theta), \chi_2(\theta)\in C^\infty(\R)$ we choose
 \[\begin{split}
g^{(1)}& = -{k\over i\tau}{i\tau+\lambda \over k}\chi_1(\theta),\\
G^{(1)}& = {k\over i\tau}(\chi_2(\theta),\sin\theta,-\cos\theta)^t,\\
g^{(2)}& =-{k\over i\tau}{i\tau+\lambda \over k}\chi_2(\theta),\\
G^{(2)}& = {k\over i\tau} (\chi_1(\theta),\sin\theta,-\cos\theta)^t.
\end{split}\]
By direct computation, we obtain $b^{(1)}=b^{(2)}=0$. From \eqref{eqn:CSW_R_T_smallness} and \eqref{eqn:CSW_B_S}, we can obtain $s^{(1)}, s^{(2)}\in H^1_{-\delta,0}(T)$. Then $e^{i\tau r}s^{(j)}$ is the only solution to 
$$
    e^{\tau x_1}( -\Delta - k^2 +\tilde{q}_j)e^{-\tau x_1}(e^{-i\tau r} s^{(j)}) = 0\ \hbox{in $T$}.
$$
By uniqueness result (\ref{eqn:CSW_R_T_smallness}), we obtain $s^{(1)}=s^{(2)}=0$. Hence $y^{(1)} = y^{(2)}= 0.$

To obtain $(\vE, \vH)$, observe that $\|S^{(j)}\|_{L^2_{-\delta}(T)}$ ($j=1,2$) is bounded in $\tau$ and 
\[\begin{split}
B^{(1)}&=\tau a^{(2)}D(\varphi+i\psi)+\tau D(\varphi+i\psi)\times A^{(2)}+D a^{(2)}+D\times A^{(2)}+kA^{(1)},\\
B^{(2)}&=\tau a^{(1)}D(\varphi+i\psi) - \tau D(\varphi+i\psi)\times A^{(1)}+D a^{(1)} - D\times A^{(1)}+kA^{(2)}.
\end{split}\]
By the $\tau$-dependence of above choice of $g(\theta)$, it is easy to see that the leading term in $L^2$ norm of $B^{(1)}$ and $B^{(2)}$ are out of the first term 
\[\begin{split}\tau  a^{(2)}D(\varphi+i\psi)&=-\tau\frac{e^{i\lambda z}}{\sqrt{2ir}}\chi_2(\theta)D(\varphi+i\psi)+O_\tau(1),\\
\tau  a^{(1)}D(\varphi+i\psi)&=-\tau\frac{e^{i\lambda z}}{\sqrt{2ir}}\chi_1(\theta)D(\varphi+i\psi)+O_\tau(1),
\end{split}\] 
respectively. This finishes the proof of \eqref{cylinerMax}. 
\end{proof}

For the other case of nonlinear phase $\varphi+i\psi=-\log\overline z$, although we are able to construct the CGO solution to the reduced Schr\"odinger equation of the form (\ref{eqn:CSW_Z_Sch}) with uniqueness, the calculation suggests that we are unable to pick a proper $A$ such that the $b^{(1)}=b^{(2)}=0$.

However, the construction here is on a bounded domain $\Omega$ and the assumption that the parameters $\mu$ and $\varepsilon$ are asymptotic to constants $\mu_0, \varepsilon_0$ is artificial, namely, we could extend the parameters to a larger domain $\wt\Omega\supset\supset\Omega$ such that $\mu, \varepsilon\in C^2_0(\wt\Omega)$ and construct the solution there. As a result, $k=0$ in our calculation and the corresponding $b^{(1)}=b^{(2)}=0$ if we choose $G^{(1)}=G^{(2)}=0$, by \eqref{eqn:CSW_b_12}.
Meanwhile, $S$ is still bounded in $H^1_0(\wt \Omega)$ with respect to $\tau$. Again, by uniqueness, we have $s^{(1)}=s^{(2)}=0$. 
Finally, for arbitrary $C^1$ functions $\chi_1(\theta)$ and $\chi_2(\theta)$, we can have
\[\begin{split}
B^{(1)}&=\tau \frac{e^{i\lambda \overline z}}{\overline z\sqrt{2ir}}\chi_2(\theta)(i,\cos\theta,\sin\theta)^t+O_\tau(1),\\
B^{(2)}&=\tau\frac{e^{i\lambda \overline z}}{\overline z\sqrt{2ir}}\chi_1(\theta)(i,\cos\theta,\sin\theta)^t+O_\tau(1),
\end{split}\]
by letting  $g^{(1)}(\theta)=\chi_1(\theta)$ and $g^{(2)}(\theta)=\chi_2(\theta)$. Together, this implies that we obtain the CGO solutions to Maxwell's equations on $\wt\Omega$ given by
\begin{align}\label{s_EH}
\vE(x_1,r,\theta)&=\tau\gamma^{-1/2} \frac{e^{i\lambda(x_1-ir)}}{(x_1-ir)^{\tau+1}\sqrt{2ir}}\chi_1(\theta)\left(\begin{array}{c}i \\\cos\theta \\\sin\theta\end{array}\right)+O_\tau(1),\notag \\
\vH(x_1,r,\theta)&=\tau\mu^{-1/2} \frac{e^{i\lambda(x_1-ir)}}{(x_1-ir)^{\tau+1}\sqrt{2ir}}\chi_2(\theta)\left(\begin{array}{c}i \\\cos\theta \\\sin\theta\end{array}\right)+O_\tau(1).
\end{align}
Note that these solutions are not oscillating with respect to $r$.

\subsection{Accelerating nondiffracting beams}
We conclude this section with more remarks discussing on the properties of the solutions we constructed above. These beams exhibit shape-preserving, nondiffracting acceleration.

In (\ref{cylinerMax}), if we choose $\chi_1(\theta)=\chi_2(\theta)=-e^{i\rho\theta}$ for $\rho>0$, then for $\tau$ large, the electromagnetic wave packets behave as
\begin{equation}\label{eqn:Ndiffra-acce}
\begin{split}
\vE(x,t) & \sim \tau\gamma^{-1/2}  \frac{e^{(-\tau+i\lambda)(x_1+ir)}}{\sqrt{2ir}}\left(\begin{array}{c} -i \\\cos\theta \\\sin\theta\end{array}\right)e^{i\rho\theta-i\omega t},\\
\vH(x,t) & \sim \tau\mu^{-1/2}  \frac{e^{(-\tau+i\lambda)(x_1+ir)}}{\sqrt{2ir}}\left(\begin{array}{c} -i \\\cos\theta \\\sin\theta\end{array}\right)e^{i\rho\theta-i\omega t}.
\end{split}
\end{equation}
To explain the non-diffracting property, when looking at the transverse profile of $\vE$ and $\vH$ (without considering $e^{i\rho\theta-i\omega t}$), we observe that both real part and imaginary part have {\bf intensity} independent of $\theta$ if $\gamma$ and $\mu$ are independent of $\theta$. 

On the other hand, we observe that the first component $\textbf{E}_1$ electric wave is shape-preserving as shown in Figure \ref{Kelvinfig} a) while the second and the third components  are not shape-preserving on their own, as shown in Figure \ref{Kelvinfig} b) and c). It suggests a power shift from $\textbf{E}_2$ component to $\textbf{E}_3$ component, which was also observed in \cite{KBNS} and explained as the rotation of the fields to stay normal to the beam bending trajectory.

For the spherical wave (\ref{s_EH}), we observe that there is no oscillation with respect to the radius $r$ as shown in Figure \ref{plt_sphere_E1}. However, the intensity profile still preserves along every circle whenever $x_1$ is fixed.

\begin{figure}[htbp] 
      \centering
       \begin{minipage}[b]{.3\textwidth} 
           (a)\includegraphics[width=2 in]{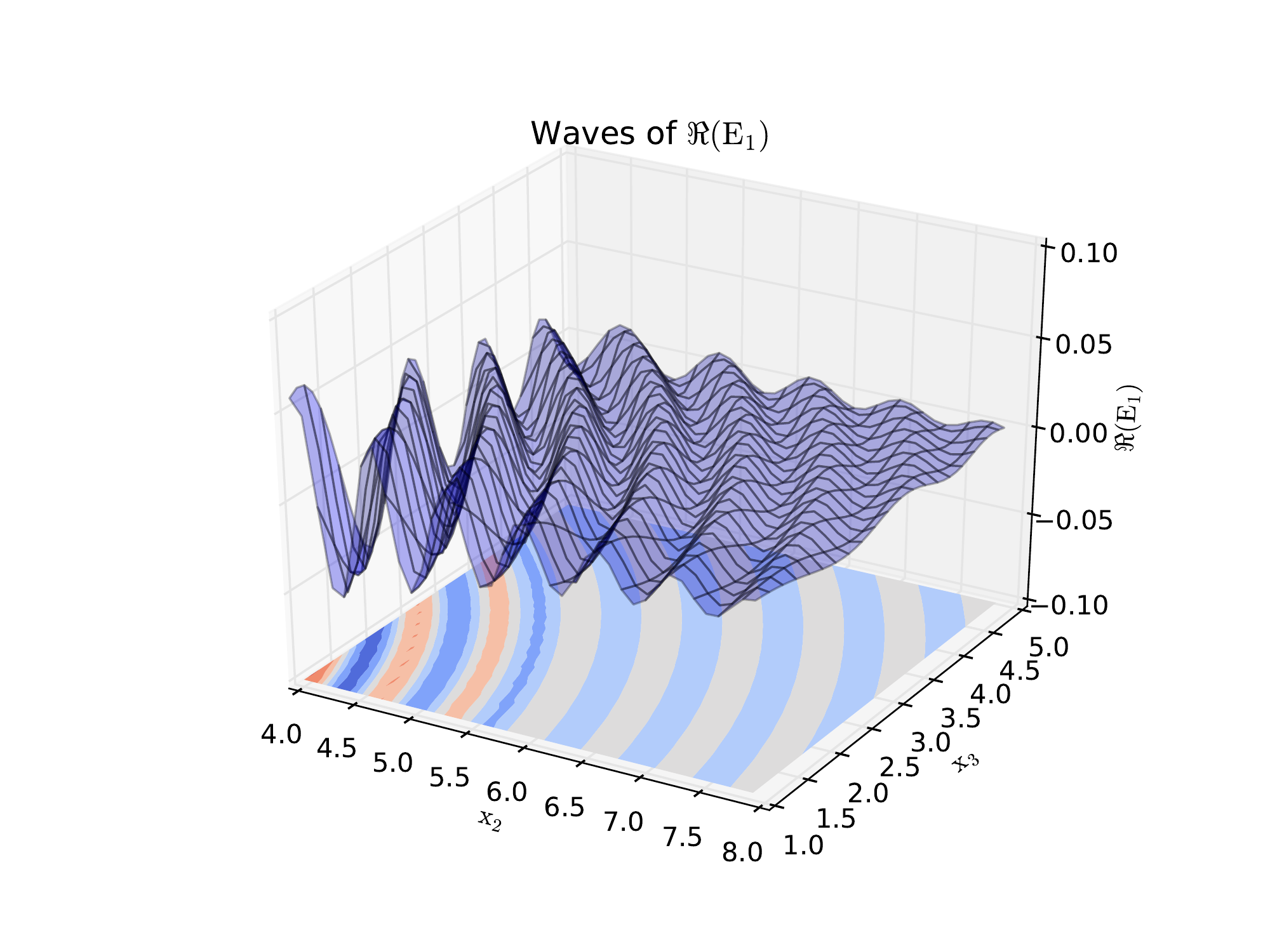} 
 
             \end{minipage}  
        \hspace{1cm}
       \begin{minipage}[b]{.3\textwidth} 
          (b)\includegraphics[width=2 in]{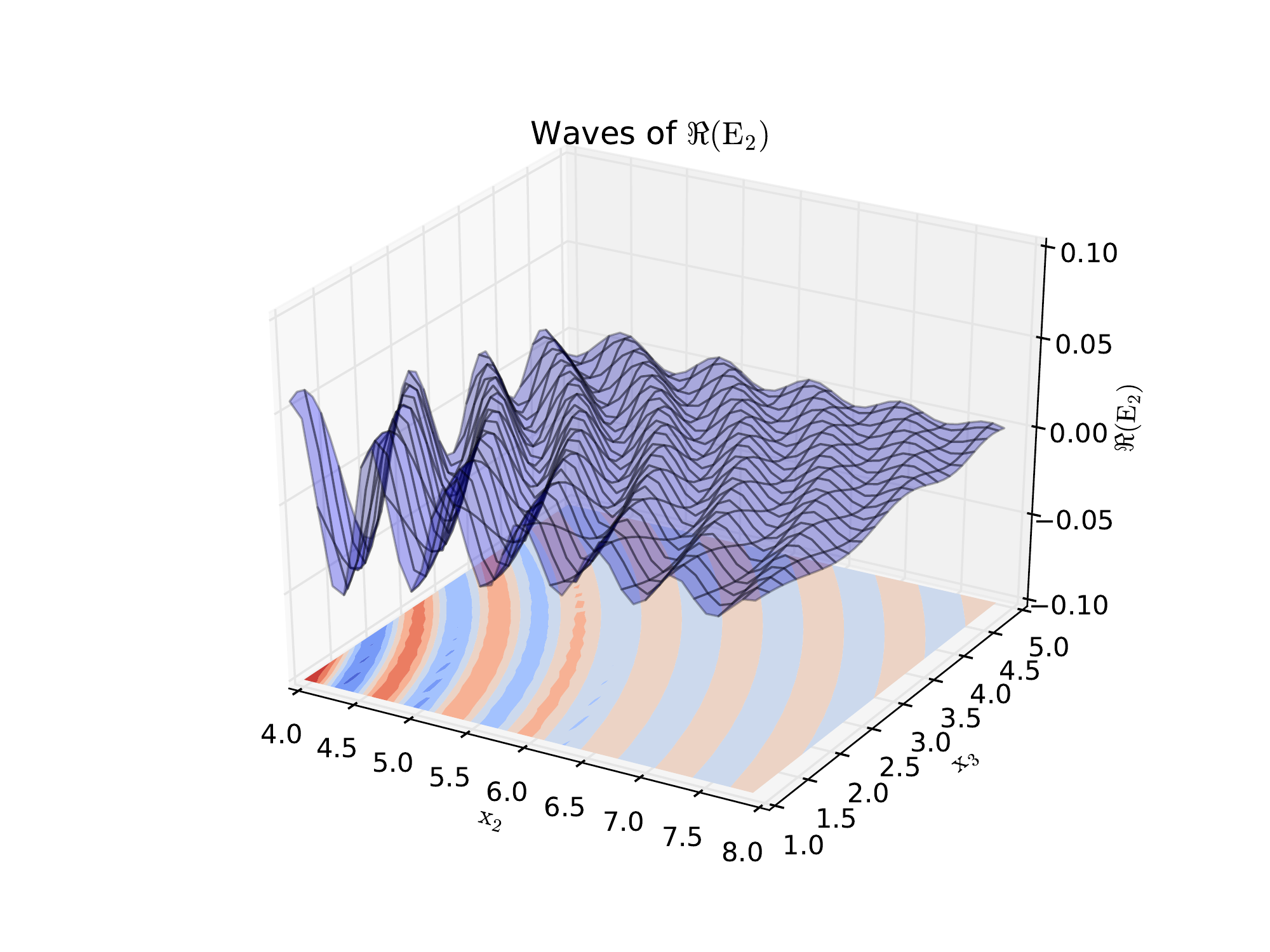} 
 
       \end{minipage} 
      \vspace{0.8cm}
      
      \begin{minipage}[b]{.3\textwidth} 
          (c)\includegraphics[width=2 in]{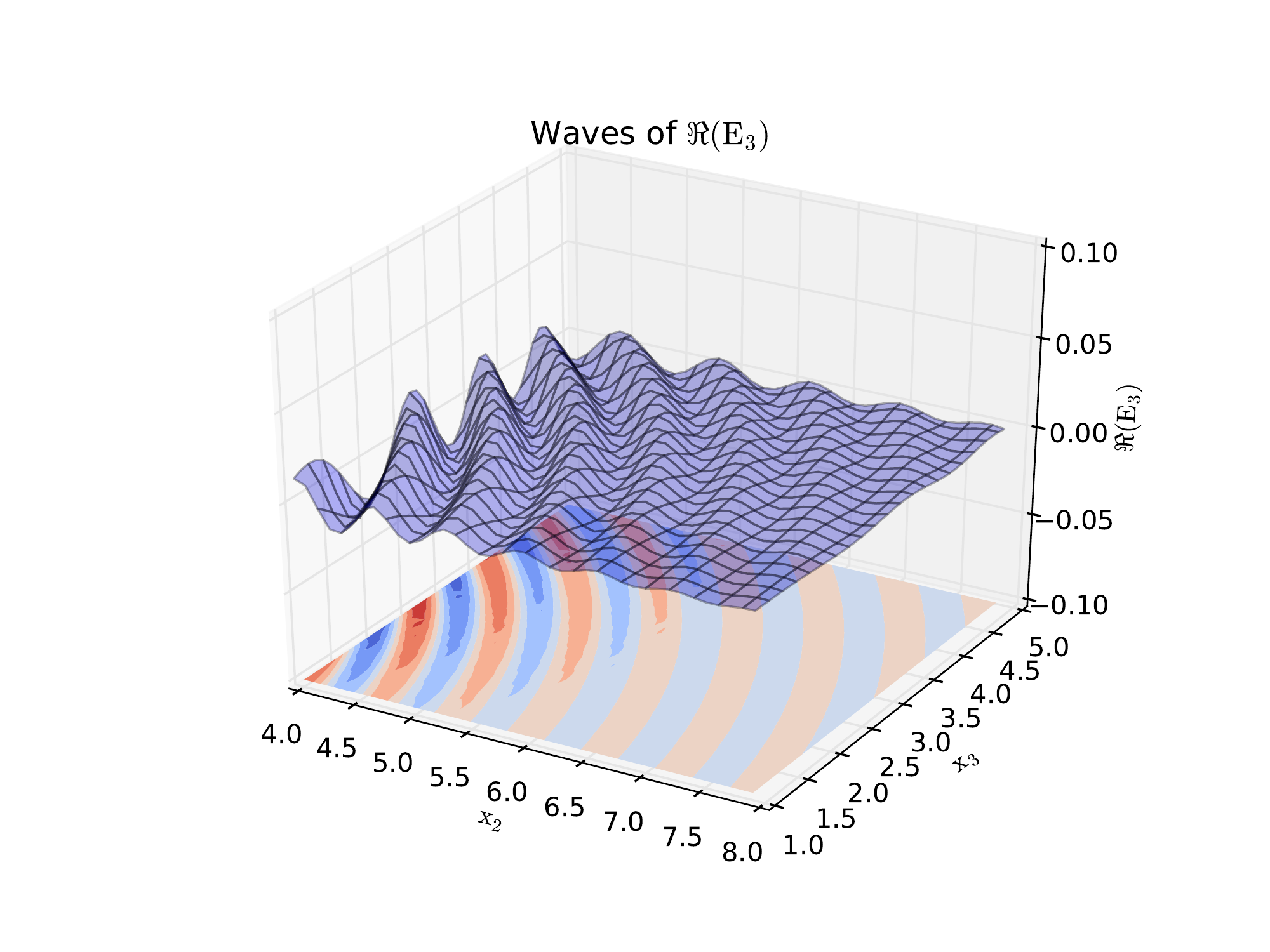} 
 
      \end{minipage} 
        \hspace{1cm}
      \begin{minipage}[b]{.3\textwidth} 
          (d)\includegraphics[width=2 in]{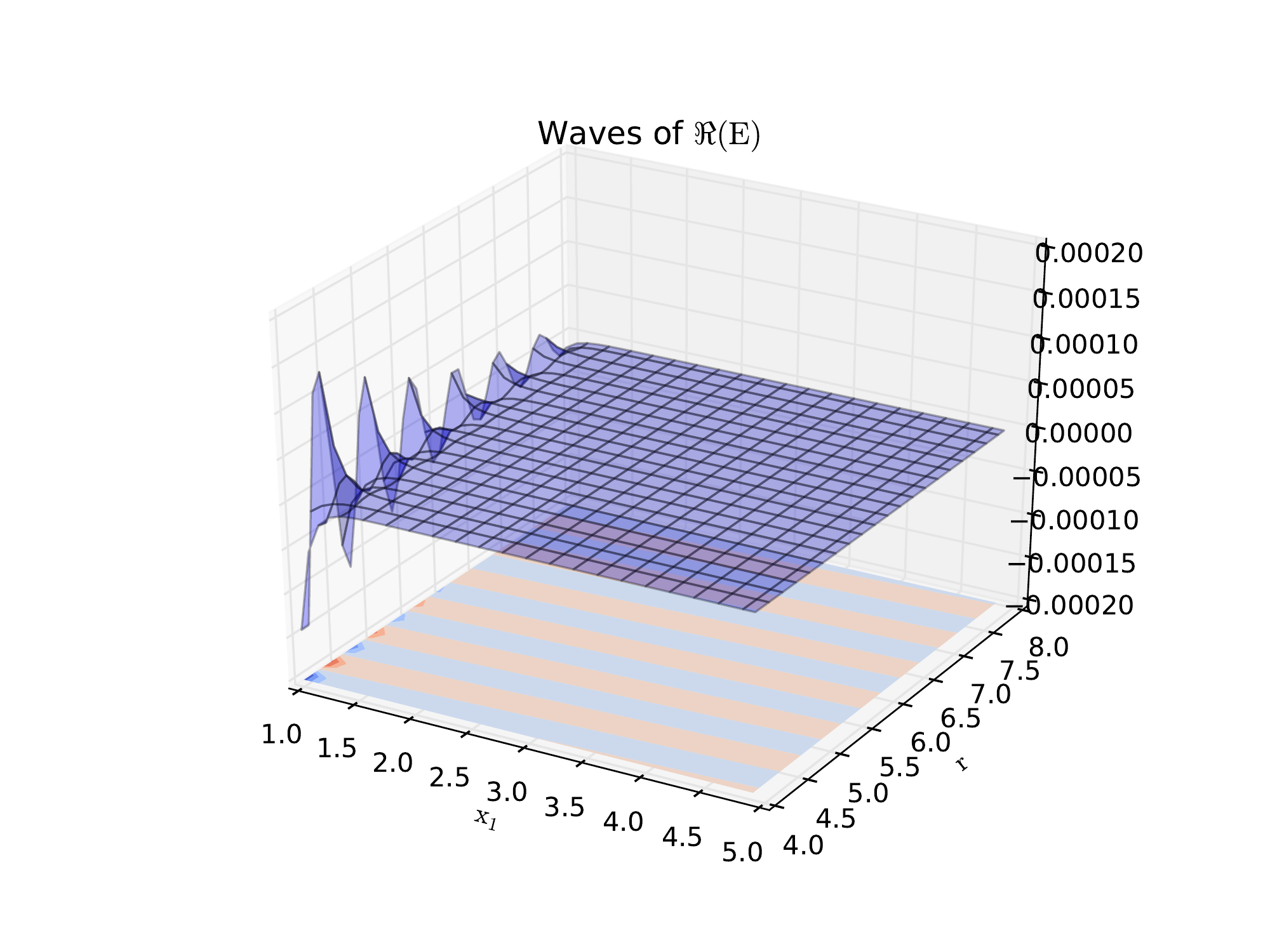} 

      \end{minipage}                     
                     
    \caption{Accelerating wave packets when $\tau = 10$ and $\lambda=0.5$ (a) The beam $\textbf{E}_1$ exhibiting shape-preserving. The trajectory is confined in a circle on the plane $\{x_1=0\}$. (b) and (c), the figures show that the propagating beams $\textbf{E}_2$ and $\textbf{E}_3$ are along a circular trajectory with varying intensity. (d) The intensity of $\textbf{E}$ decreases and oscillates when $r$ is growing. }\label{Kelvinfig} 
\end{figure}

\begin{figure}[h]
    \begin{minipage}[b]{.3\textwidth}
         \includegraphics[width=2in]{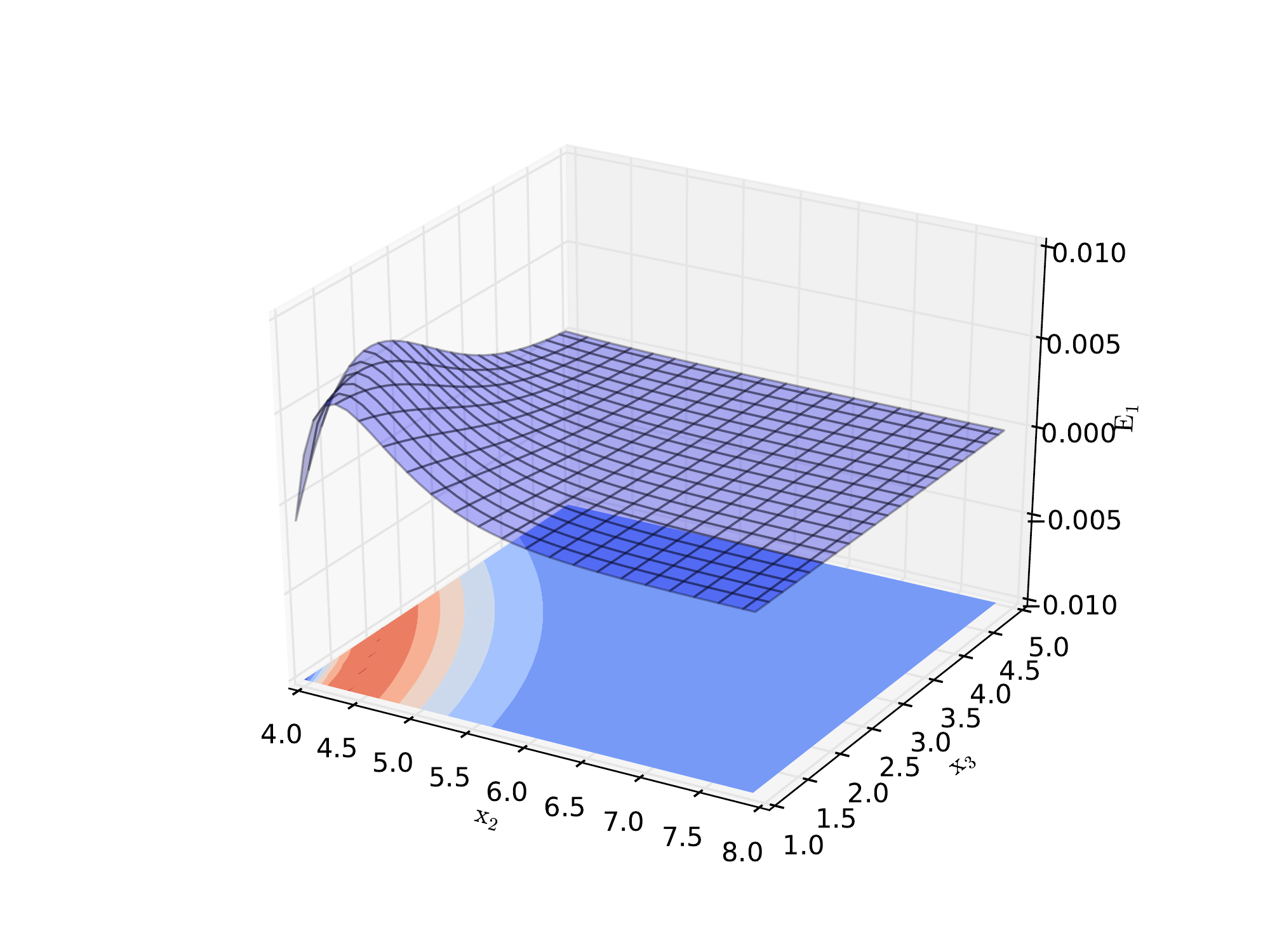}               
    \end{minipage}
       \caption{Intensity profiles of the spherical accelerating waves $\textbf{E}_1$ in (\ref{s_EH}) when $\tau=9$ and $\lambda =0.5$ on the plane $\{x_1=2\}$.} \label{plt_sphere_E1}
\end{figure}



\begin{remark}
The reduction of Maxwell's equations to the vectorial Schr\"odinger equation was inspired (see \cite{OS}) by the physically referred auxiliary Hertz potentials (also known as Sommerfeld potentials) introduced to handle equations in vacuum (homogeneous) background. That is to write the electric and magnetic fields as
\[\begin{split}
\vE  & =-\nabla\times\Pi_m-\frac{1}{i\omega\varepsilon_0}(\nabla\nabla\cdot+k^2)\Pi_e,\\
\vH & =\nabla\times\Pi_e-\frac{1}{i\omega\mu_0}(\nabla\nabla\cdot+k^2)\Pi_m,
\end{split}\]
where the Hertz vector potential $\Pi_{m,e}$ satisfies the vector Schr\"odinger equation 
\[(-\Delta-k^2)\Pi_{m,e}=0.\]
Note that these potentials are given by $Z$ through 
\[Z=\left(\begin{array}{c}-\frac i{\omega\varepsilon_0^{1/2}}D\cdot\Pi_e\\ i\varepsilon_0^{1/2}\Pi_m\\ -\frac i{\omega\mu_0^{1/2}}D\cdot\Pi_m \\i\mu_0^{1/2}\Pi_e\end{array}\right).\]

When $\Pi_m=\psi \hat e_1$ where $\psi(r,\theta)$ is a scalar function independent of $x_1$, and $\Pi_e=0$, we have
\[
     \vE = \LC -\hat{r}{\partial_\theta \over r} + \hat{\theta}\partial_r \RC \psi, 
\]
where $\hat r=(0,\cos\theta,\sin\theta)^t$ and $\hat \theta=(0,-\sin\theta,\cos\theta)^t$ are the corresponding radial and angular unit vectors. We recover the TM-polarized propagation mode. 

%
It was analyzed in \cite{BAKS} that, for their nondiffracting accelerating beam, the radial $\hat{r}$ component plays the dominant role and preserves the shape of propagation. We observe that there is a similar behavior for our solution, the radial component of \eqref{eqn:Ndiffra-acce} also dominates the propagation direction of the waves. 

\end{remark}

\begin{remark}
We would like to point out that the CGO solutions are constructed on a bounded domain with boundary conditions, although we do not restrict ourselves to a specific domain. In another word, this is not a freely propagating solution (a half-space solution) sent initially from a plane like $x_3=0$ physically as in \cite{AMMK,AB,KBNS,SC1}. 
\end{remark}

\section{Kelvin transform based accelerating beams}\label{section2d}
In this section, we apply Kelvin transform to construct the solutions to Maxwell's equations in the physical space, that is, the parameter are constants, which is the situation discussed in most physical papers. We show these beams also have shape-preserving property.

\subsection{Kelvin transform and transformation law for Maxwell's equations}

First, we recall the transformation law for Maxwell's equations. Suppose $\tilde x=F(x)$ is a diffeomorphism on $\overline{\R^3}$, where $\overline{\R^3}:=\R^3\cup \{\infty\}$ and let $(\textbf{E}, \textbf{H})$ be the solution to Maxwell's equations
\begin{equation}
\label{eqn:Max_const}
\nabla\times \vE-i\omega\mu_0 \vH = 0,\quad \nabla\times \vH+i\omega\varepsilon_0 \vE=0.  
\end{equation}
If we define the {\em push-forward} of the electric and magnetic fields $\textbf{E}$ and $\textbf{H}$ by $F$ as 
\[
\begin{split}\wt \vE(\tilde x) & =F_*\vE:=DF^{-1}(\tilde x)\vE\circ F^{-1}(\tilde x),\\
 \wt \vH(\tilde x) & =F_*\vH:=DF^{-1}(\tilde x)\vH\circ F^{-1}(\tilde x),
 \end{split}
 \]
where $DF$ denotes the Jacobian matrix of the transformation,
then these push-forward fields $\wt\vE$ and $\wt\vH$ satisfy Maxwell's equations in the space of $\tilde x$
\begin{equation}
\label{eqn:Max_pf}
\nabla\times \wt\vE-i\omega\tilde\mu \wt\vH = 0,\quad \nabla\times \wt\vH+i\omega\tilde\varepsilon\wt\vE=0
\end{equation}
with {\em push-forward} parameters
\[
\begin{split}
\tilde\mu(\tilde x) & =F_*\mu_0:=\frac{DF~\mu_0~DF^t}{|\det(DF)|}\circ F^{-1}(\tilde x),\\
\tilde\varepsilon(\tilde x) & =F_*\varepsilon_0:=\frac{DF~\varepsilon_0~DF^t}{|\det(DF)|}\circ F^{-1}(\tilde x).
\end{split}
\]
This invariance is just another presentation of independence of choice of coordinates for Maxwell's equations. 

In order to construct accelerating beams in the free space with constant $\mu_0$ and $\varepsilon_0$, which we call the {\em physical space}, we consider the push-forward system living in the space with $\tilde\mu$ and $\tilde\varepsilon$, which we call the {\em virtual space}, by a special $F$, known as the Kelvin transform. It is the following reflection map with respect to the sphere $\partial B(0,R)$ of radius $R>0$,
\[\tilde x=K(x):=\frac{R^2}{|x|^2}(x_1,x_2,x_3),\]
where $|x|^2=\sum_{j=1}^3x_j^2$. 
It maps any sphere through origin to a hyperplane and satisfies $K^{-1}=K$.  A two dimensional configuration is shown in Figure \ref{fig:K2}. 
\begin{figure}[htbp] 
   \centering
   \includegraphics[width=3 in]{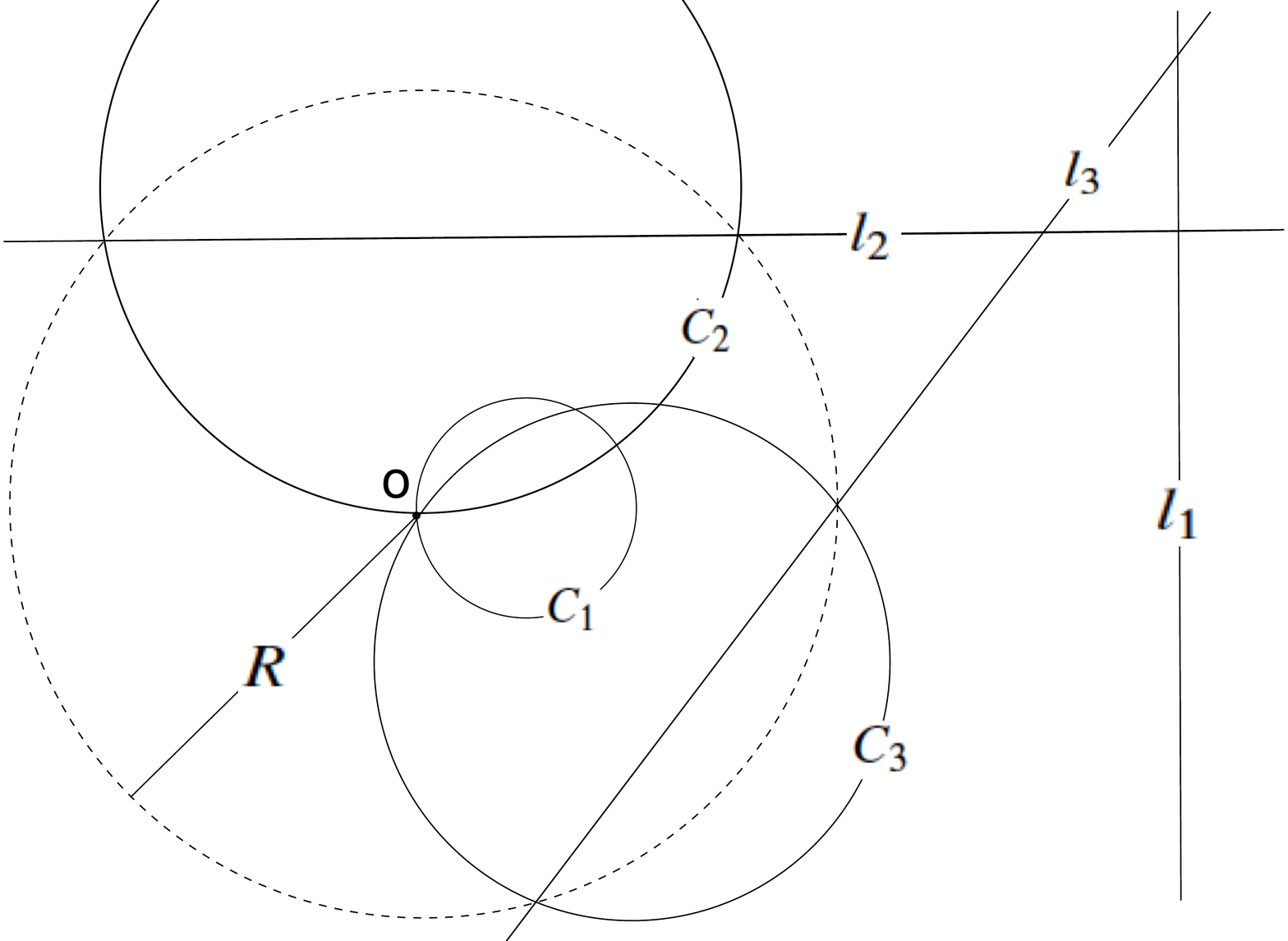} 
   \caption{\small The reflection Kelvin transform $K$ in a plane: The image of the circle (sphere) $C_j$ ($j=1,2,3$) through the origin is the straight line (hyperplane) $l_j$ and vice versa since $K=K^{-1}$.}
   \label{fig:K2}
\end{figure}
Correspondingly, away from the origin, we have
\begin{equation}\label{eqn:DK} 
    DK(x)=\frac{R^2}{|x|^2}(I-2\widehat r\,\widehat r^t),
\end{equation}
where $\widehat r:=x/|x|$ is the radial unit vector and $I$ denotes the identity matrix.
Therefore,  $|\det(DK)|=R^6/|x|^6$ and 
\[\frac{(DK)~(DK)^t}{|\det(DK)|}=\frac{|x|^2}{R^2}=\frac{R^2}{|\tilde x|^2}\]
which gives 
\begin{equation}\label{eqn:pf_mu_eps}
    \tilde\mu(\tilde x)=\frac{R^2}{|\tilde x|^2}\mu_0,\quad \tilde\varepsilon(\tilde x)=\frac{R^2}{|\tilde x|^2}\varepsilon_0.
\end{equation}



\subsection{CGO solutions in the virtual space}
 
We aim to use the CGO solution to Maxwell's equations \eqref{eqn:Max_pf} in an annulus $\mathcal A:=\left\{\tilde x~|~\frac{R^2}{L}<|\tilde x|<L\right\}$ for some $L>4R$ with $(\tilde\mu,\tilde\varepsilon)$ given by \eqref{eqn:pf_mu_eps}.  Note that $K^{-1}(\mathcal A)=\mathcal A$. 

Let $\phi$ be a $C^\infty(\R)$ cut-off function with $\phi\equiv 1$ on $\left(\frac{R^2}{L},L\right)$ and $\phi\equiv 0$ on $\R\backslash \left[\frac{R^2}{2L}, 2L\right]$. Set 
\[\mu(\tilde x)=\phi(|\tilde x|)\tilde\mu(\tilde x),\quad \varepsilon(\tilde x)=\phi(|\tilde x|)\tilde\varepsilon(\tilde x).\]
Then we want to construct CGO solutions to Maxwell's equations with parameters $(\mu, \varepsilon)$ in $\Omega:=B(0,2L)$. Note that now $\mu,\varepsilon\in C^\infty_0(\Omega)$. 
Following the steps in Section \ref{section3D}, one looks for the CGO solution to the Schr\"odinger equation
\[(P+W)(P-W^t)Z=(-\Delta+Q)Z=0.\]

By Remark \ref{rmk:linear}, we have for $\zeta\in\mathbb C^3$ satisfying $\zeta\cdot\zeta=0$ and $|\zeta|\sim\tau$ sufficiently large, 
\begin{equation}\label{eqn:CGO_sch_linear}
Z=e^{\zeta\cdot\tilde x}\left(A_\zeta+R_\zeta(\tilde x)\right)
\end{equation}
for some constant vector $A_\zeta$ of order $O_{\tau\rightarrow\infty}(1)$, where $R_\zeta\in H^2_{-\delta}$.
Moreover, 
\[\|R_\zeta\|_{H^s_{-\delta}}\leq C|\zeta|^{s-1},\qquad \delta>\frac{1}{2}.\]

For our purpose, let $\rho>0$ be the spatial propagation frequency in $\tilde x_3$ direction. We choose 
\[\zeta=\frac{1}{2}\left(\left(\begin{array}{c}-\tau \\ \tau \\0\end{array}\right)+i\left(\begin{array}{c}\sqrt{\tau^2-\rho^2} \\ \sqrt{\tau^2-\rho^2} \\\sqrt 2\rho\end{array}\right)\right)\]
whose norm is $|\zeta|^2=\tau^2$. 

Naturally, given $Z$ as \eqref{eqn:CGO_sch_linear}, we have
\[Y  = (P-W^t)e^{\zeta\cdot\tilde x}(A_\zeta+R_\zeta(\tilde x)) = e^{\zeta\cdot \tilde x}\left(B_\zeta+S_\zeta(\tilde x)\right),\]
where
\begin{align*}
    & B_\zeta  = P(-i\zeta)A_\zeta,\quad\\
    & S_\zeta = -W^t(\tilde x)A_\zeta+(P+P(-i\zeta))R_\zeta(\tilde x)-W^t(\tilde x)R_\zeta(\tilde x).
\end{align*}
Here $P(-i\zeta)$ denotes the matrix of the form \eqref{eqn:PofD} but with $D$ replaced by $-i\zeta$. 
If we choose 
\[A_\zeta=\frac 1\tau\left(\zeta\cdot \mathbf a, \vec0, \zeta\cdot\mathbf b, \vec0\right)^t\]
for arbitrary $\mathbf a, \mathbf b\in\R^3$, it immediately gives
\[\begin{split}
B_\zeta & =\left(0,\frac{ -i\zeta\cdot \mathbf b}\tau\zeta^t, 0, \frac{ -i\zeta\cdot\mathbf a}\tau\zeta^t\right)^t\\
&=-i\tau\left(0,~(\widehat\zeta_0\cdot\mathbf b){\widehat\zeta_0}^t,~0,~(\widehat\zeta_0\cdot\mathbf a){\widehat\zeta_0}^t\right)^t+O_{\tau\rightarrow\infty}(1),
\end{split}\]
where 
\[\widehat\zeta_0=\lim_{\tau\rightarrow\infty}\zeta/\tau=\frac{1}{2}\left(\left(\begin{array}{c}-1 \\1 \\0\end{array}\right)+i\left(\begin{array}{c}1 \\ 1 \\0\end{array}\right)\right).\]
Since the first and the fifth components of $B_\zeta$ vanish, applying the uniqueness addressed in Remark \ref{rmk:linear}, we obtain that there exists a solution to Maxwell's equations in $\Omega$ given by
\[
\begin{split}
    \wt\vE(\tilde x) 
    & = -i\tau\varepsilon^{-1/2}(\tilde{x})\tilde{\textbf{e}}(\tilde{x}),\\
    \wt\vH(\tilde x) 
    & = -i\tau\mu^{-1/2}(\tilde{x}) \tilde{\textbf{h}}(\tilde{x}).
\end{split}
\]
Here we denote
\[\begin{split}
    \wt\ve(\tilde x) & = e^{\frac{1}{2}(-\tau(\tilde{x}_1 - \tilde{x}_2) + i\sqrt{\tau^2-\rho^2}(\tilde{x}_1 + \tilde{x}_2))}e^{i\rho\tilde x_3}\left[(\widehat\zeta_0\cdot\mathbf a)\widehat\zeta_0+O_\tau(\tau^{-1})\right],\\
    \wt\vh(\tilde x) & =  e^{\frac{1}{2}(-\tau(\tilde{x}_1 - \tilde{x}_2) + i\sqrt{\tau^2-\rho^2}(\tilde{x}_1 + \tilde{x}_2))}e^{i\rho\tilde x_3}\left[(\widehat\zeta_0\cdot\mathbf b)\widehat\zeta_0+O_\tau(\tau^{-1})\right].
\end{split}\]
Here $O_\tau(\tau^{-1})$ denotes vector functions whose $L^2(\Omega)$ norm is bounded with respect to $\tau$.

\subsection{Accelerating beams in the physical space}

In the annulus $\mathcal A$, from (\ref{eqn:pf_mu_eps}), we have for $\tau$ large
\begin{equation*}
\wt\vE(\tilde x) =-i\tau\varepsilon_0^{-1/2}\frac{|\tilde x|}{R}\wt\ve(\tilde x),\quad
\wt\vH(\tilde x) =-i\tau\mu_0^{-1/2}\frac{|\tilde x|}{R}\wt\vh(\tilde x).
\end{equation*}
Note that $\tilde{\textbf{e}}(\tilde{x})$ and $\tilde{\textbf{h}}(\tilde{x})$ are the near planewave parts in the virtual space. 
The other observation is that both $\wt\vE$ and $\wt\vH$ are almost perpendicular to $\tilde x_3$ for $\tau$ large, due to the choice of $\zeta$. 

By \eqref{eqn:DK} and that $K=K^{-1}$, we obtain the solutions to the original Maxwell's equations in free space
\begin{equation}\label{3:kelvinsol}
\begin{split}
\vE(x) & =K_*\wt\vE=-i\tau\varepsilon_0^{-1/2}\frac{R^3}{|x|^3}\left[I-2\widehat r\widehat r^t\right]\wt\ve(K(x)),\\
\vH(x) & =K_*\wt\vH=-i\tau\mu_0^{-1/2}\frac{R^3}{|x|^3}\left[I-2\widehat r\widehat r^t\right]\wt\vh(K(x)).
\end{split}
\end{equation}
The formula suggests the following properties of such a wave.

\begin{enumerate}
\item 
In Figure \ref{Physical_e1} a), the transverse profile $\tilde{\textbf{e}}_1$ of the near plane-wave part at $\tilde x_3=0$ in the virtual space is shown. One notices that the peaks and valleys of the oscillation reside on lines $\tilde x_1 - \tilde x_2=c$ (while the exponential decay is along the perpendicular lines). In the virtual space, these peaks and valleys propagate straight in $\tilde x_3$ direction. These peak and valley planes $\{\tilde x_1 - \tilde x_2=c\}$ are mapped to spheres passing through the origin in the physical space. In Figure \ref{Physical_e1} b), by using Kelvin transform with respect to the sphere of radius $R=5$, we depict the peak propagation due to the factor 
$
     R^3/|x|^3
$ 
(without multiplication by $[I-\widehat r\widehat r^t]$). Note that the ``lobes" feature ``shrinks" as propagating away from the $x_1x_2$-plane $\left\{x_3=0\right\}$ while the intensity increases due to $\frac{R^3}{|x|^3}$. Hence the intensity keeps shifting from the tail towards the main ``lobe", suggesting an overly self-healing due to the wave acceleration.

\item Multiplication by the matrix $[I-2\widehat r\widehat r^t]$, that is, we consider the solutions in (\ref{3:kelvinsol}), which corresponds to a reflection of the electric and magnetic fields and does not change the norm of any real vector. The real and imaginary intensities of the fields are preserved. 
\item The construction applies to a lossy system where the conductivity $\sigma$ is a positive constant, in which case we only need to replace $\varepsilon_0$ by the complex number $\gamma_0=\varepsilon_0 + i \omega^{-1}\sigma_0$. Above self-healing property still exists and compensates the energy loss during the propagation. 
\end{enumerate}


\begin{figure}[h] 
   \begin{minipage}[b]{.4\textwidth}
           (a)\includegraphics[width=2in]{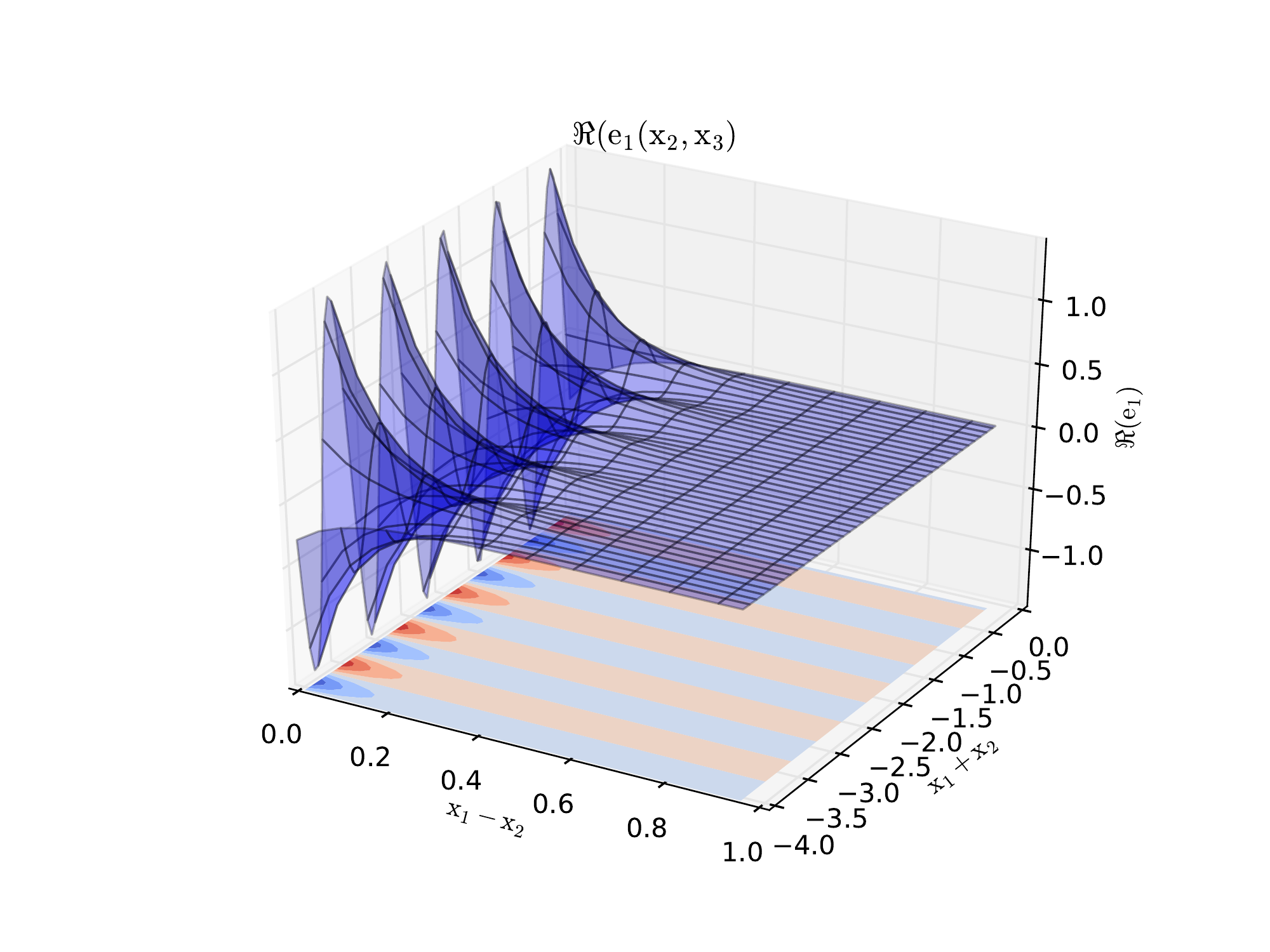}                      
      \end{minipage}
          \hspace{0.8cm}
    \begin{minipage}[b]{.4\textwidth}
      (b)\includegraphics[width=2in]{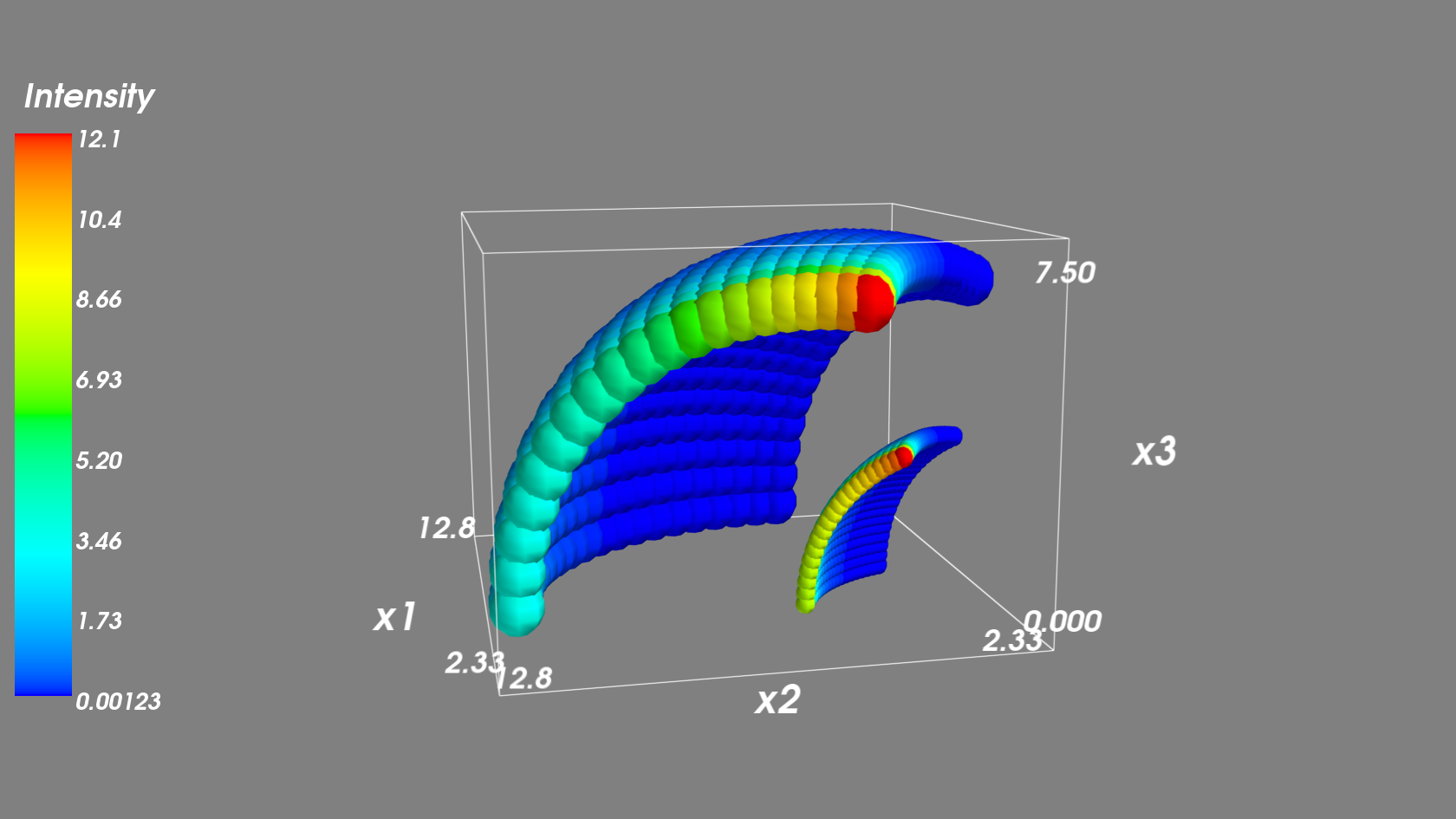}                      
        \end{minipage}
        \caption{(a) The cross section of the near plane-wave $\Re\tilde{{e}}_1(K(x))$ in the virtual space at $\tilde{x}_3=0$. (b) Beams $R^3 |x|^{-3}  \Re{\tilde{e}}_1(K(x))$ with the choice of $\textbf{a} = (-1/\sqrt{2}, -1/\sqrt{2}, 0)^t$, $\tau = 4$ and $\sqrt{\tau^2-\rho^2} = 3$ on the spheres $(x_1-a)^2+(x_2-a)^2+x_3^2 = 2 a^2$ when $a=50/3\pi, 150/17\pi$ in the physical space.}
                 \label{Physical_e1}
\end{figure}

\end{document}